\documentclass[12pt, reqno, a4paper]{amsart}
\usepackage{amsfonts, amsmath, amsthm, amstext, amssymb}
\usepackage{mathtools}
\usepackage{url}
\usepackage{paralist}
\usepackage[english]{babel}
\usepackage{framed}
\usepackage{amsrefs}
\usepackage[marginratio=1:1,height=660pt,width=410pt,tmargin=100pt]{geometry}
\usepackage{lipsum}
\usepackage{layout}
\usepackage{hyperref}
\usepackage{mathrsfs}

\DeclareMathOperator{\CC} {\mathbb{C}}

\DeclareMathOperator{\QQ} {\mathbb{Q}}
\DeclareMathOperator{\RR} {\mathbb{R}}

\DeclareMathOperator{\ZZ} {\mathbb{Z}}

\DeclareMathOperator{\bG} {\bf{G}}

\DeclareMathOperator{\bGL} {\bf{GL}}

\DeclareMathOperator{\Gal} {Gal}
\DeclareMathOperator{\Aut} {Aut}

\DeclareMathOperator{\Spec} {Spec}

\DeclareMathOperator{\poly} {poly}
\DeclareMathOperator{\rk} {rk}

\newtheorem{lemma}{Lemma}[section]
\newtheorem{thm}[lemma]{Theorem}

\theoremstyle{definition}\newtheorem{definition}[lemma]{Definition}
\theoremstyle{definition}\newtheorem{remark}[lemma]{Remark}
\theoremstyle{definition}
\theoremstyle{definition}
\theoremstyle{definition}

\newcommand{\mb}[1]{\mathbb{#1}}
\newcommand{\mc}[1]{\mathcal{#1}}
\newcommand{\mf}[1]{\mathfrak{#1}}
\newcommand{\ms}[1]{\mathscr{#1}}
\newcommand{\ra}{\rightarrow}

\newcommand{\ol}[1]{\overline{#1}}

\newcommand{\tb}[1]{\textbf{#1}}

\begin{document}

\title[Polylog-size coverings of rational points]{Tower of coverings and polylog-size coverings of rational points}
\author{Kenneth Chung Tak Chiu}
\email{kennethct.chiu@alumni.utoronto.ca}
\address{Department of Mathematics and New Cornerstone Science Laboratory, The University of Hong Kong.}
\begin{abstract}
We prove that a smooth irreducible proper variety over a number field admits polylogarithmic-size coverings of rational points of bounded height if it admits a tower of finite \'{e}tale coverings satisfying several conditions. 
One of the conditions is a height comparison on twists of the finite \'{e}tale covers.
We prove that these conditions hold for abelian varieties. 
\end{abstract}

\maketitle

\section{Introduction}\label{Introduction}
A major theme in Diophantine geometry is the study of the growth of the number of rational points of bounded height on varieties over a number field.
This paper is a continuation of a line of works by Ellenberg-Lawrence-Venkatesh \cite{ELV23}, Brunebarbe-Maculan \cite{BM23}, and the author \cite{Chiu22}.
The main result in Ellenberg-Lawrence-Venkatesh was stated in the setting of variations of Hodge structures, while the two other papers were respectively in the settings of varieties with large fundamental groups and variations of mixed Hodge structures. In the special case for a proper Shimura variety $\mc{S}$, it states that the number of rational points on $\mc{S}$ with height $\leq B$ is $O_{\mc{S}}(B^{\varepsilon})$ for any $\varepsilon>0$. 
The first step was to cover the rational points with height $\leq B$ by $O_{\mc{S}}(B^{\varepsilon})$ many irreducible subvarieties of codimension at least 1,
in other words, $\mc{S}$ admits a subpolynomial-size covering of rational points of bounded height.
This was done by constructing certain tower of finite \'{e}tale covers of $\mc{S}$ and applying Broberg's theorem to the twists of these covers by certain cocycles in Galois cohomology.
Broberg's theorem \cite{Bro04}, which builds on fundamental ideas of Bombieri-Pila \cite{BP89} and Heath-Brown \cite{Hea02}, is a uniform bound for the number of hypersurfaces of bounded degree needed to contain the rational points of bounded height on an arbitrary irreducible closed subvariety of fixed dimension and degree.

In this paper, we prove that a smooth irreducible proper variety over a number field admits polylogarithmic-size coverings of rational points of bounded height in the sense of Definition \ref{definition polylogarithmically sparse covering} below if it admits a tower satisfying certain assumptions, namely a height comparison on twists in the sense of Definition \ref{height comparison on twists} below, and polynomial growth of the size of certain Galois cohomology as the level of the tower goes up.
We prove that any abelian variety over a number field admits a tower satisfying these assumptions. 
For comparison, the Bombieri-Lang conjecture states that the variety admits a finite covering of rational points if it is of general-type.

\subsection{Polylogarithmic-size coverings}\label{Statements of results}
Let $\mc{X}$ be a smooth irreducible proper variety over a number field $E$ of dimension $n$. 
We write $\alpha=\poly_{X}(\beta)$ as a shorthand for $\alpha\leq P_X(\beta)$, where $X\mapsto P_X$ is a universally fixed mapping whose values are univariate polynomials with positive coefficients.

\begin{definition}\label{definition polylogarithmically sparse covering}
The variety $\mc{X}$ has a \emph{polylogarithmic-size covering of rational points} if for any $B>e$, the $E$-points of $\mathcal{X}$ with exponential Weil height at most $B$ can be covered by 
$\poly_{\mathcal{X}}(\log B)$ irreducible subvarieties over $E$ of dimension $< \dim \mathcal{X}$ and degree 
$\poly_{\mathcal{X}}(\log B)$ with respect to an ample line bundle on $\mc{X}$.
\end{definition}

Let $\{\mc{X}_m\}_{m\in \mathbb{N}}$ be an infinite tower of finite \'{e}tale geometrically Galois coverings $f_m:\mc{X}_m\to \mc{X}$ of $\mc{X}$.
Let $G_m:=\Aut((f_m)_{\ol{E}})$ be the geometric deck group of $f_m$.

\begin{definition}\label{Gamma prime power structure}
The tower $\{\mc{X}_m\}$ has \emph{$\Gamma_{p^m}$-structures} if there exist an infinite subgroup $\Gamma$ of $\textbf{GL}_k(\ZZ)$, where $k$ is a positive integer, and a prime number $p$ such that for all $m$, the geometric deck group $G_m$ is isomorphic to the image $\Gamma_{p^m}$ of $\Gamma$ under the mod $p^m$ map $\bGL_k(\ZZ)\rightarrow \bGL_k(\ZZ/p^m\ZZ)$.
\end{definition}

Suppose $\{\mc{X}_m\}$ is equipped with $\Gamma_{p^m}$-structures. 
For each $m$, we have a Galois action $\theta_m:\Gal(\ol{E}/E)\to \Aut(G_m)$ defined by 
$\theta_m(\sigma)(\gamma)=\sigma_{\mc{X}_m}\circ \gamma\circ\sigma_{\mc{X}_m}^{-1}$,
where $\sigma_{\mc{X}_m}$ is the automorphism on $(\mc{X}_m)_{\ol{E}}$ induced by $\sigma^{-1}$.
This action $\theta_m$ is a group homomorphism.

\begin{definition}\label{exceptional primes}
A finite set $T$ of primes of $E$ is \emph{exceptional} for the tower $\{\mc{X}_m\}$ if 
\begin{itemize}
\item $T$ contains all primes above $p$;
\item $\mathcal{X}$ spreads out to a smooth proper integral model $\ms{X}$ over the ring $\mathcal{O}_{E,T}$ of $T$-integers in $E$;
\item each $f_m$ extends to a finite \'{e}tale surjective morphism $\ms{X}_m\to \ms{X}$, where $\ms{X}_m$ is a smooth $\mc{O}_{E,T}$-model of $\mc{X}_m$.
\end{itemize}
\end{definition}

Suppose further that $\{\mc{X}_m\}$ has an exceptional set $T$ of primes of $E$.
For all $m$, the Galois action $\theta_m$ is unramified outside $T$, see Lemma \ref{Galois action unramification}.
Let $E_{T}$ be the maximal extension of $E$ unramified outside $T$.
The Galois action $\theta_m$ induces an action $\Gal(E_T/E)\to \Aut(G_m)$.

We denote by $H^1(\Gal(E_T/E),G_m)$ the first Galois cohomology for the action $\theta_m$.
For any $\tau=[\rho]$ in $H^1(\Gal(E_T/E), G_m)$,
let $\mc{X}_m^{\tau}$ be the corresponding \emph{twist} given by \cite[Theorem 4.5.2]{Poonen17}.
There exists an isomorphism $\lambda: (\mc{X}_m)_{E_T}\to (\mc{X}_m^{\tau})_{E_T}$ such that 
$\lambda^{\sigma}:=\sigma_{\mc{X}_m^{\tau}}\circ\lambda\circ\sigma_{\mc{X}_m}^{-1}=\lambda\circ \rho(\sigma)$ 
for any $\sigma\in \Gal(E_T/E)$, see \cite[Remark 4.5.4]{Poonen17}.
Let $f_m^{\tau}:\mc{X}_m^{\tau}\ra \mc{X}$ be the descended finite morphism characterized by $(f_m^{\tau})_{E_T}\circ \lambda = (f_m)_{E_T}$.

For any line bundle $L$ on a variety $\mc{X}$, let $H_L$ be the exponential Weil height with respect to $L$,
i.e. $H_L(P)=\exp(h_L(P))$, where $h_L$ is the logarithmic Weil height with respect to $L$.

\begin{definition}\label{height comparison on twists}
Let $\mf{T}:=\{\mc{X}_m\}$ be a tower with $\Gamma_{p^m}$-structures such that there exists an exceptional set $T$ of primes. The tower $\mf{T}$ satisfies \emph{height comparison on twists} if there exist a line bundle $L$ on $\mc{X}$ and constants $C_1,C_2,C_3>0$ depending only on $\mf{T}$ and $T$ such that for any positive integer $m$ and any $\tau\in H^1(\Gal(E_T/E),G_m)$, there exists an $\mc{O}_{E}$-model $\ms{X}_m^{\tau}$ of the twist $\mc{X}_m^{\tau}$ with a projective embedding $\iota_m^{\tau}:\ms{X}_m^{\tau}\hookrightarrow \mathbb{P}^{M(m,\tau)}_{\mc{O}_{E}}$ such that the following conditions hold: 
\begin{itemize}
\item $M(m,\tau) \leq \poly_{\mc{X},\mf{T},T}(\deg f_m)$,
\item  $f_m^{\tau*}L\simeq \iota_m^{\tau*}\mc{O}(1)_{E}$,
\item  for any $B>e$ and any $P\in \mc{X}_m^{\tau}(E)$ with $H_{L}(f_m^{\tau}(P))\leq B$, we have 
 $$h_{\iota_m^{\tau*}\ol{\mc{O}(1)}}(P)< C_1+C_2\log B+ C_3(\deg f_m)^{1/n},$$ 
where $h_{\iota_m^{\tau*}\ol{\mc{O}(1)}}$ is the logarithmic Arakelov height with respect to the Hermitian line bundle $\iota_m^{\tau*}\ol{\mc{O}(1)}$ equipped with the Fubini-Study metrics.
\end{itemize}
\end{definition}

In Section \ref{Proof of main theorem}, we will prove the following main result:

\begin{thm}\label{main theorem} 
Let $\mc{X}$ be a smooth irreducible proper variety over a number field $E$ of dimension $n$. 
Suppose $\mc{X}$ has a tower $\mf{T}:=\{\mc{X}_m\}$ with $\Gamma_{p^m}$-structures. 
Suppose there exists an exceptional set $T$ of primes for $\mf{T}$. 
If $\mf{T}$ satisfies height comparison on twists, and if the cardinality of $H^1(\Gal(E_T/E),G_m)$ is at most $\poly_{\mc{X},\mf{T},T}(|\Gamma_{p^m}|)$, 
then $\mc{X}$ has a polylogarithmic-size covering of rational points. 
\end{thm}

\begin{remark}\label{remark for Shimura variety}
In an informal discussion, Emmanuel Ullmo suggested the use of fibered powers of a Hecke correspondence to prove height comparison for towers of Shimura varieties. 
This will be a later work.
The use of fibered powers of a Hecke correspondence will be analogous to the use of Koizumi's theorem in the case for abelian varieties in Section \ref{Proof of Neron's theorem}.
The existence of an exceptional set of primes will require \cite[Theorem 6.2]{BST24}.
The assumption about the cardinality of $H^1(\Gal(E_T/E),G_m)$ can possibly be verified in certain cases using the theory of pro-$p$ groups. 
\end{remark}

\begin{remark}\label{Remark on hyperbolicity}
It is expected that Shimura varieties have at most polylogarithmically many integral points. 
This will need new inputs from hyperbolicity. 
In recent works \cite{Mem22, Mem25, Mem26} by Memariansorkhabi, volume estimates are used to prove subpolynomial counting bounds for \emph{rational points} on \emph{non-proper} ball quotients and Hilbert modular varieties.
\end{remark}

In Section \ref{Proof of Neron's theorem}, we will prove the following:

\begin{thm}\label{abelian variety, height comparison}
Let $\mc{A}$ be an abelian variety over a number field $E$. Then $\mc{A}$ admits a tower $\mf{T}$ with $\Gamma_{p^m}$-structures and exceptional set $T$ of primes such that height comparison on twists holds and the cardinality of $H^1(\Gal(E_T/E),G_m)$ is at most $\poly_{\mc{A},\mf{T},T}(|\Gamma_{p^m}|)$.
\end{thm}
 
The proof of Theorem \ref{abelian variety, height comparison} uses Koizumi's theorem \cite{Koi76}, twisted trace, geometry of numbers, canonical metrics introduced by Zhang \cite{Zha95}, and results in de Jong and Shokrieh \cite{dJS22}.

Theorem \ref{main theorem} and Theorem \ref{abelian variety, height comparison} together imply that any  abelian variety over a number field has a polylogarithmic-size covering of rational points. 
In particular, any elliptic curve over a number field has at most polylogarithmically many rational points.
N\'{e}ron's theorem \cite{Ner65} states that the number of rational points of height at most $B$ on an abelian variety $\mc{A}$ is $\ll_{\mc{A}} (\log B)^{r/2}$, where $r$ is the rank of $\mc{A}$. 
A proof of N\'{e}ron's theorem using Theorem \ref{main theorem} and Theorem \ref{abelian variety, height comparison} will need new inputs.

\subsection{Intuition behind Theorem \ref{main theorem} and its proof}\label{intuition}
Every rational points can be lifted to one of the twists of the finite \'{e}tale cover \cite[p. 22]{Sko01}, so it suffices to bound the number of twists, estimate the growth of rational points on the twist, and compare the height of a rational point on the twist and the height of its image in $\mc{X}$.
The conditional bound for the size of $H^1(\Gal(E_T/E),G_m)$ provides a bound for the number of twists. 
Let $n$ be the dimension of $\mc{X}$.
Let $d$ be the degree of the twist of finite \'{e}tale cover with respect to some embedding into $\mb{P}^M_E$.
If one tries to apply Broberg's theorem \cite{Bro04},  the number of hypersurfaces of bounded degree needed to contain the rational points with height at most $B$ on the twist is at most $O_{M,d,\varepsilon}(B^{(n+1)d^{-1/n}+\varepsilon})$ for any $\varepsilon >0$. 
By taking $d\sim (\log B)^n$, this estimate is polynomial in $(\log B)\cdot B^{\varepsilon}$ if the implicit constant is polynomial in $d$.
To compute the implicit constant and get rid of the factor $B^{\varepsilon}$, we use an Arakelov version of Broberg's theorem instead. 
Such a version was proved by Huayi Chen \cite{Chen12a, Chen12b}. The bound in this version is polynomial in $B^{(n+1)d^{-1/n}{(1+\varepsilon)}}$, which becomes finite when we fix $\varepsilon$ and take $d\sim (\log B)^n$.
However, the constant in the bound was computed to be exponential in $M$.
We use an idea in a previous paper \cite{Chiu24} by the author to replace the constant by a polynomial in $M$ (and thus in $d$) using Dudek's theorem on primes between cubes \cite{Dud16}.
Alternatively, it may be possible to apply \cite[Theorem 1.1]{Chiu24} directly, but one has to compute the implicit constants in the bound for the size of the complex cellular cover, see (1) in \emph{op. cit.} and \cite{BN19}.

\subsection{Acknowledgement}
The author would like to thank Gal Binyamini, Cong Ding, Ziyang Gao, Soheil Memarian, Ngaiming Mok, Carlo Pagano, Ananth Shankar, Jacob Tsimerman, Emmanuel Ullmo, Kwok Kin Wong, Jiacheng Xia, Sai-Kee Yeung, Asif Zaman and Shou-Wu Zhang for discussions and comments.
AI was not used in any stage of writing from Section \ref{Introduction} to Section \ref{Proof of main theorem}. 
The author discussed Section \ref{Proof of Neron's theorem} with Gemini 3.1 Pro and GPT-5.6 Sol Pro in the early stage of proof structuring.
The final proofs were written by the author. 
The author takes full responsibility for the content of the paper.
The author was partially supported by the New Cornerstone Science Foundation through the New Cornerstone Investigator Program awarded to Professor Xuhua He.
The author was partially supported by the Glorious Sun Postdoctoral Fellowship.

\section{Arakelov geometry and height functions}
We first recall a few notions in \cite{BG06}, \cite{BGS94}, and \cite{Chen12b}. 
Let $K$ be a finite extension of $\QQ$.
An arithmetic variety over the ring of integers $\mathcal{O}_K$ is a flat integral scheme of finite type over $\Spec \mathcal{O}_K$.
Let $\ms{X}$ be a regular projective arithmetic variety over $\mc{O}_K$.
Let $X:=\ms{X}_K$.
Let $\overline{\mathscr{L}}$ be a Hermitian line bundle on $\ms{X}$.
We denote by $\overline{\mathscr{L}}^\vee$ the dual of $\overline{\mathscr{L}}$.
Let $Z$ be an integral closed subscheme of $X$ of dimension $p$. 
Let $\ms{Z}$ be the Zariski closure of $Z$ in $\ms{X}$.

\begin{definition}[{\cite[\S 3.1.4]{BGS94}}]
The \emph{(normalized) logarithmic Arakelov height} of $Z$ with respect to $\overline{\mathscr{L}}$ is defined to be 
$$h_{\overline{\mathscr{L}}}(Z):=\frac{1}{[K:\QQ]}\widehat{\deg}(\widehat{c_1} (\overline{\mathscr{L}})^{p+1} \cdot [\mathscr{Z}]),$$
where $\widehat{c_1} (\overline{\mathscr{L}})$ is the arithmetic first Chern class of $\overline{\mathscr{L}}$ defined in \cite[\S 2.1.2]{BGS94},
and $\widehat{\deg}$ is the arithmetic degree defined in \cite[\S 2.3.4]{BGS94}.
\end{definition}

\begin{definition}
The \emph{Arakelov degree} of a Hermitian vector bundle $\overline{E}$ on $\Spec \mathcal{O}_K$ is defined as
$$\widehat{\deg}(\overline{E}):=\log \# (E/(\mathcal{O}_Ks_1+\cdots +\mathcal{O}_K s_r))-\frac{1}{2}\sum_{\sigma: K\hookrightarrow \CC}\log \det (\langle s_i,s_j\rangle_{\sigma}),$$
where $(s_1,\dots, s_r)\in E^r$ forms a basis of $E_K$ over $K$ and $r$ is the rank of $E$.
\end{definition}

\begin{definition}
Let $\overline{E}$ be a non-zero Hermitian vector bundle on $\Spec \mathcal{O}_K$. The \emph{slope} of $\overline{E}$ is 
$$\widehat{\mu} (\overline{E}):=\frac{1}{[K:\QQ]} \frac{\widehat{\deg}(\overline{E})}{\rk(E)}.$$
Denote by $\widehat{\mu}_{\max}(\overline{E})$ the maximum of the slopes of all non-zero Hermitian subbundles of $\overline{E}$.
\end{definition}

\begin{definition}
Let $P\in X(K)$. The \emph{(normalized) logarithmic Arakelov height} of $P$ with respect to $\overline{\mathscr{L}}$ is defined to be 
$$h_{\overline{\mathscr{L}}}(P):=\frac{1}{[K:\QQ]}\widehat{\deg}(\widehat{c_1} (\overline{\mathscr{L}})\cdot [\mathscr{P}]),$$
where $\mathscr{P}$ is the image of the section $\varepsilon_P: \Spec \mc{O}_K\ra \ms{X}$ attached to $P$.
By \cite[3.1.2.1]{BGS94}, it is equal to the slope of $\varepsilon_P^*\ol{\ms{L}}$.
By \cite[\S 3.1.4]{BGS94}, this height is invariant under base change, so this definition extends to points in $X(\ol{K})$.
\end{definition}

\begin{definition}\label{height of a point}
Let $P\in X(K)$. The \emph{exponential Arakelov height} $H_{\ol{\ms{L}}}(P)$ of $P$ with respect to $\overline{\mathscr{L}}$ is defined to be $\exp ([K:\QQ] h_{\ol{\ms{L}}}(P))$.
\end{definition}

Let $\ol{\mc{O}(1)}$ be the Hermitian line bundle on the projective space $\mb{P}^M_{\mc{O}_K}$ equipped with the Fubini-Study metrics. 
Write $n_v:=[K_v: \QQ_{v_0}]$, where $v_0$ is the place below a place $v$.
The logarithmic Arakelov height with respect to $\ol{\mc{O}(1)}$ of a point $[x_0:\cdots: x_{M}]$ in $\mb{P}^{M}(K)$ is given by
$$h_{\ol{\mc{O}(1)}}([x_0,\dots, x_M])=\frac{1}{[K:\QQ]}\left(\sum_{v\nmid \infty } n_v\log \max_{0\leq j\leq M}|x_j|_v+\sum_{v\mid\infty} n_v\log \left(\sum_{j=0}^M|x_j|_v^2\right)^{1/2}\right).$$
Let $M_K$ be the set of all places of $K$.
The logarithmic Weil height of $[x_0,\dots, x_M]$ is given by 
$$h([x_0,\dots, x_{M}])=\frac{1}{[K:\QQ]}\sum_{v\in M_K}n_v\log \max_{0\leq j\leq M}|x_j|_v.$$

\section{Hypersurface coverings of rational points}
In this section, we replace a constant which appeared in \cite{Chen12b} by a polynomial in $M$ as mentioned in Section \ref{intuition}.
Let $K$ be a finite extension of $\QQ$.
Let $M$ be a positive integer.
Let $\iota: X\hookrightarrow \mathbb{P}_K^M$ be a closed immersion, where $X$ is an irreducible subvariety of dimension $n$ and degree $d$.
Let $\mathscr{X}$ be the Zariski closure of $X$ in $\mathbb{P}^M_{\mathcal{O}_K}$.
Let $\overline{\mathscr{L}}:= \overline{\mathcal{O}_{\mathscr{X}}(1)}$ be the pullback of the Hermitian line bundle $\overline{\mathcal{O}(1)}$ to $\mathscr{X}$ equipped with the pullbacks of the Fubini-Study metrics at the Archimedean places.
Let $S(X,B)$ be the set of $K$-points of $X$ with exponential Arakelov height $H_{\overline{\mathscr{L}}}(P) \leq B$.
For any maximal ideal $\mathfrak{p}$ of $\mathcal{O}_K$, denote by $\mathbb{F}_{\mathfrak{p}}$ its residue field, and by $N_{\mathfrak{p}}$ the cardinality of $\mathbb{F}_{\mathfrak{p}}$.  

Since $\mathscr{X}$ is a proper $\mathcal{O}_K$-scheme, any point in $X(K)$ extends to an $\mathcal{O}_K$-point.
For any maximal ideal $\mathfrak{p}$ of $\mathcal{O}_K$ and any $a\in \ZZ_{>0}$, denote by $A^{(a)}_{\mathfrak{p}}$ the Artinian local ring $\mathcal{O}_{K,\mathfrak{p}}/\mathfrak{p}^a\mathcal{O}_{K,\mathfrak{p}}$.
For any $\eta\in \mathscr{X}(A^{(a)}_{\mathfrak{p}})$, denote by $S(X,B,\eta)$ the set of points in $S(X,B)$ whose reduction modulo $\mathfrak{p}^a$ coincides with $\eta$.
For any maximal ideal $\mathfrak{p}$ of $\mathcal{O}_K$, define 
$$S_1(X,B, \mathfrak{p}):= \bigcup_{\text{regular }\xi\in \mathscr{X}(\mathbb{F}_{\mathfrak{p}})} S(X, B, \xi).$$

Let $v$ be the least common multiple of $1,2,\dots, [K:\QQ]$.
Let $S_1(X,B)$ be the set of regular $K$-points $P$ of $X$ with exponential Arakelov height $H_{\overline{\mathscr{L}}}(P) \leq B$.

\begin{thm}\label{Arakelov Bombieri-Pila}
Let $\varepsilon >0$. 
Let $D$ be an integer such that
$$D>\max\left\{(\varepsilon^{-1}+1)(d^{-1/n}(n+3)/2 +d-2), 2(M-n)(d-1)+n+2\right\}.$$
Let
\begin{equation}\label{choosing B}
B_0:=\max\left\{\exp(3\exp(33.3)d^{1/n}\frac{nv}{n+1}(1+\varepsilon)^{-1}), \exp(\varepsilon)\right\}.
\end{equation}
For any $B> B_0$, the set $S_1(X,B)$ is covered by $N$ hypersurfaces defined over $K$ of degree $D$, none of which contain $X$,
where  
$$N\leq\poly_{n,K}(M,d, \varepsilon^{-1})\cdot (M+1)^{[K:\QQ](1+\varepsilon)(n+1)d^{-1/n}}B^{(1+\varepsilon)(n+1)d^{-1/n}}.$$
\end{thm}

\begin{proof}
In the last paragraph of the proof of Theorem 4.2 in \cite{Chen12b}, Chen proved that if $D>2(M-n)(d-1)+n+2$ and if
$$\frac{\log B}{[K:\QQ]}<  \frac{n!}{d(2n+2)^{n+1}} h_{\overline{\mathscr{L}}} (X)-\frac{3}{2} \log (M+1)-2^n,$$
then 
$S_1(X,B)$ is contained in a hypersurface of degree $D$ in $\mathbb{P}^M_K$ which does not contain $X$.
It suffices to prove the case when
$$\frac{\log B}{[K:\QQ]}\geq \frac{n!}{d(2n+2)^{n+1}} h_{\overline{\mathscr{L}}} (X)-\frac{3}{2} \log (M+1)-2^n.$$
Let $N_0\in (1,\infty)$ such that
\begin{equation}\label{N_0}
\log N_0= (1+\varepsilon)d^{-1/n}\frac{n+1}{nv}\left(\log B+[K:\QQ] \log (M+1)\right).
\end{equation}
Let $r$ be the integral part of the number
$$\frac{(M-n)(d-1)\log B+((M-n)h_{\overline{\mathscr{L}}}(X)+C_3)[K:\QQ]}{\log N_0}+1,$$
where $C_3$ is the constant defined as follows in Notation 19 of Chen's paper \cite{Chen12b}:
let $\overline{\mathscr{E}}$ be the trivial Hermitian vector bundle of rank $M+1$ over $\Spec \mathcal{O}_K$, let $S^d(\overline{\mathscr{E}}^\vee)$ be the $d$-th symmetric power of $\overline{\mathscr{E}}^\vee$, let 
\begin{align*}
C_1:&=(n+2)\widehat{\mu}_{\max}(S^d(\overline{\mathscr{E}}^\vee))+\frac{1}{2}(n+2)\log \rk (S^d \mathscr{E})
\\&\quad+\frac{d}{2} \log ((n+2)(M-n))+\frac{d}{2}(n+1)\log (M+1),
\end{align*} 
$$C_2:= \frac{M-n}{2}\log \rk (S^d\mathscr{E})+\frac{1}{2}\log \rk (\Lambda^{M-n}\mathscr{E})+\log \sqrt{(M-n)!}+ (M-n)\log d,$$
$$C_3:= (M-n)C_1+C_2.$$
We have 
$$r\leq \frac{A_1\log B+A_2}{\log N_0}+1,$$
where $$A_1=(M-n)(d-1)+\frac{(2n+2)^{n+1}}{n!} (M-n)d$$
and 
$$A_2=[K:\QQ]\left(C_3+(M-n)\frac{(2n+2)^{n+1}}{n!} d\left(\frac{3}{2}\log (M+1)+2^n\right)\right).$$
Since $\log B\geq \varepsilon$, we have
$$r< A_3:=v\frac{A_1+\varepsilon^{-1}A_2}{(1+\varepsilon)d^{-1/n} (n+1)/n}+1.$$

By \eqref{choosing B} and \eqref{N_0}, we have $N_0^{1/3}>e^{e^{33.3}}$, so by Dudek's Theorem \cite{Dud16},  there exist $r$ distinct prime numbers $p_1,\dots, p_r$ such that 
$$N_0\leq p_i\leq (N_0^{1/3}+i)^{3}$$ for any $i=1,\dots, r$. 
Choose for each $i$ a maximal ideal $\mathfrak{p}_i$ of $\mathcal{O}_K$ lying over $p_i$.
By \cite[Lemma 4.1]{Chen12b}, 
$$S_1(X,B)=\bigcup_{i=1}^r S_1(X,B, \mathfrak{p}_i).$$
For any $i$, write $N_{\mathfrak{p}_i}=p_i^{f_i}$.
Let $a_i:=v/f_i$.
Let $\xi$ be a regular $\mathbb{F}_{\mathfrak{p}_i}$-point of $\mathscr{X}_{\mathbb{F}_{\mathfrak{p}_i}}$.
We have
\begin{align*}
\log N_{\mathfrak{p}_i}^{a_i}=\log p_i^{a_if_i}=v \log p_i
&\geq v\log N_0
\\&= (1+\varepsilon) d^{-1/n}\frac{n+1}{n}(\log B + [K:\QQ]\log (M+1)),
\end{align*}
so by \cite[Cor. 3.7]{Chen12b} (the case where $|J|=1$) and the assumption on $D$, for any $\eta\in \mathscr{X}(A_{\mathfrak{p}_i}^{(a_i)})$ whose reduction modulo $\mathfrak{p}_i$ is $\xi$,
there exists a hypersurface of degree $D$ of $\mathbb{P}^M_K$ not containing $X$ that contains $S(X,B, \eta)$. 
As in the proof of \cite[Theorem 4.2]{Chen12b}, every regular point in $\ms{X}(\mb{F}_{\mf{p}_i})$ has at most $N_{\mf{p}_i}^{(a_i-1)n}$ lifts to $\mathscr{X}(A_{\mathfrak{p}_i}^{(a_i)})$.
By \cite[Prop. 12.1]{GL02}, the cardinality of $\ms{X}(\mb{F}_{\mf{p}_i})$ is at most 
$\displaystyle d\sum_{k=0}^nN_{p_i}^k$.
Hence, the cardinality of $\mathscr{X}(A_{\mathfrak{p}_i}^{(a_i)})$ is at most 
$$d(n+1) N_{\mathfrak{p}_i}^{a_in}=d(n+1)p_i^{f_ia_in}=d(n+1)p_i^{vn}\leq d(n+1) (N_0^{1/3}+i)^{3vn}.$$
Hence, $S_1(X,B,\mathfrak{p_i})$ is covered by at most $d(n+1) (N_0^{1/3}+i)^{3vn}$ hypersurfaces of degree $D$ of $\mathbb{P}^M_K$ not containing $X$.
Therefore, $S_1(X,B)$ is covered by at most 
\begin{align*}
\sum_{i=1}^r d(n+1) (N_0^{1/3}+i)^{3vn}
&\leq rd(n+1)(N_0^{1/3}+r)^{3vn}
\\&\leq \poly_{n,K}(r,d) \cdot (M+1)^{[K:\QQ](1+\varepsilon)(n+1)d^{-1/n}}B^{(1+\varepsilon)(n+1)d^{-1/n}}
\end{align*}
hypersurfaces of degree $D$ of $\mathbb{P}^M_K$ not containing $X$.

It suffices to prove that $r\leq \poly_{n,K}(M,d, \varepsilon^{-1})$.
From \cite[(4.6)]{Bos01}, we know 
$$\widehat{\mu}_{\max}(S^d(\overline{\mathscr{E}}^\vee))\leq d(\widehat{\mu}_{\max}(\overline{\mathscr{E}})+2\rk (\mathscr{E})\log (\rk(\mathscr{E}))).$$ 
Since $\mathscr{E}$ is the trivial vector bundle, $\widehat{\mu}_{\max}(\overline{\mathscr{E}})=0$, 
so $\widehat{\mu}_{\max}(S^d(\overline{\mathscr{E}}^\vee))\leq \poly(M,d)$.
We also have 
$$\log \rk (S^d\mathscr{E})\leq d\log ((M+d)/d!^{1/d})\leq \poly(M,d),$$
$$\log \rk (\Lambda^{M-n}\mathscr{E})\leq (n+1)\log (M+1)\leq \poly(n,M),$$
and 
$$\log \sqrt{(M-n)!}\leq \frac{1}{2} (M-n)\log (M-n)\leq \poly(n,M).$$
Thus, $C_3\leq \poly(n, M, d)$, 
so $r<A_3\leq  \poly_{n,K}(M, d, \varepsilon^{-1})$.
\end{proof}

\section{Proof of main theorem}\label{Proof of main theorem}
We first prove a claim in the setup in the introduction. 

\begin{lemma}\label{Galois action unramification}
If $\mc{X}$ has an infinite tower $\{\mc{X}_m\}$ of finite \'{e}tale geometrically Galois coverings with $\Gamma_{p^m}$-structures, and if there exists a finite set $T$ of exceptional primes for $\{\mc{X}_m\}$, 
then the Galois action $\theta_m:\Gal(\ol{E}/E)\to \Aut(G_m)$ defined by 
$\theta_m(\sigma)(\gamma)=\sigma_{\mc{X}_m}\circ \gamma\circ\sigma_{\mc{X}_m}^{-1}$ in Section \ref{Statements of results} is unramified outside $T$ for all $m$.
\end{lemma}

\begin{proof}
Let $v$ be a prime of $E$ outside $T$.
Let $I_v\subset \Gal(\ol{E_v}/E_v)$ be the inertia group at $v$.
By \cite[Lemma 0BSW]{Stacks}, we have $I_v\subset \Gal(\ol{E_v}/M)$, 
where $M$ is the maximal unramified extension of $E_v$.
It suffices to prove that the geometric points of a fiber of $f_m$ are defined over $M$.
Since $v\not\in T$, the morphism $f_m:\mc{X}_m\to \mc{X}$ extends to a finite \'{e}tale surjective morphism 
$g_m:(\ms{X}_m)_{\mc{O}_{E_v}}\to \ms{X}_{\mc{O}_{E_v}}$  over $\mc{O}_{E_v}$.
Let $x: \Spec \mathcal{O}_{M}\to \ms{X}_{\mc{O}_{M}}$ be an $\mathcal{O}_{M}$-point, which exists since $\mc{O}_M$ is Henselian and $\mc{O}_M/\mf{m}$ is separably closed, 
where $\mf{m}$ is the maximal ideal of $\mc{O}_{M}$.
By pulling back $(g_m)_{\mathcal{O}_M}$ along $x$, we obtain a finite \'{e}tale algebra $R$ over $\mathcal{O}_{M}$.
Since $\mathcal{O}_M/\mf{m}$ is separably closed,  $R/\mf{m}R$ is isomorphic to a finite product of copies of $\mc{O}_{M}/\mf{m}$ by \cite[Lemma 00U3]{Stacks}.
Since $\mc{O}_M$ is a Henselian local ring, $R$ is isomorphic to a finite product of copies of $\mathcal{O}_M$ by \cite[Lemma 04GK]{Stacks}. 
Therefore, the fiber at $x$ over $M$ is a disjoint union of $\Spec M$.
\end{proof}

We now prove Theorem \ref{main theorem} which is restated here:

\begin{thm}\label{main theorem, shimura variety, restate} 
Suppose $\mc{X}$ has a tower $\mf{T}:=\{\mc{X}_m\}$ with $\Gamma_{p^m}$-structures. 
Suppose there exists an exceptional set $T$ of primes for $\mf{T}$. 
If $\mf{T}$ satisfies height comparison on twists, and if the cardinality of $H^1(\Gal(E_T/E),G_m)$ is at most $\poly_{\mc{X},\mf{T},T}(|\Gamma_{p^m}|)$, 
then $\mc{X}$ has a polylogarithmic-size covering of rational points. 
\end{thm}

\begin{proof}
For any $B>e$, let $c(B):=(\log B)^n$ and
choose the smallest positive integer $m=m(B)$ such that $|\Gamma_{p^m}|\geq c(B)$. 
As $\Gamma_{p^m}$ is the image of $\Gamma$ in  $\bGL_k(\ZZ/p^m\ZZ)$,
the order $|\Gamma_{p^m}|$ is at most the product of $|\Gamma_{p^{m-1}}|$ and the size of the kernel of the  map $\bGL_k(\ZZ/p^m\ZZ)\rightarrow \bGL_k(\ZZ/p^{m-1}\ZZ)$ induced by the mod $p^{m-1}$ map,
we thus have
\begin{equation}\label{bounding monodromy group, large local system}
|\Gamma_{p^m}|\leq |\Gamma_{p^{m-1}}|\cdot p^{k^2}< c(B)\cdot p^{k^2} =(\log B)^n\cdot p^{k^2}.
\end{equation}

By height comparison on twists, there exist a line bundle $L$ on $\mc{X}$ and positive constants $C_1,C_2,C_3$, depending only on $\mf{T}$ and $T$, such that for any positive integer $m$ and any $\tau\in H^1(\Gal(E_T/E),G_m)$, there exists an $\mc{O}_{E}$-model $\ms{X}_m^{\tau}$ of the twist $\mc{X}_m^{\tau}$ with a projective embedding $\iota_m^{\tau}:\ms{X}_m^{\tau}\hookrightarrow \mathbb{P}^{M(m,\tau)}_{\mc{O}_{E}}$ such that the following holds: 
\begin{itemize}
\item $M(m,\tau) \leq \poly_{\mc{X},\mf{T},T}(\deg f_m)$,
\item  $f_m^{\tau*}L\simeq \iota_m^{\tau*}\mc{O}(1)_{E}$,
\item  for any $P\in \mc{X}_m^{\tau}(E)$ and any $B>e$ with $H_{L}(f_m^{\tau}(P))\leq B$, we have 
 \begin{equation}\label{conjectural height bound}
 h_{\iota_m^{\tau*}\ol{\mc{O}(1)}}(P)< C_1+C_2\log B+ C_3(\deg f_m)^{1/n}.
 \end{equation}
\end{itemize}

We have
$$\deg (\mc{X}_m,f_m^{*}L)=\deg ((\mc{X}_m)_{\ol{E}},(f_m^{*}L)_{\ol{E}})=\deg ((\mc{X}_m^{\tau})_{\ol{E}},(f_m^{\tau*}L)_{\ol{E}})=\deg (\mc{X}_m^{\tau},f_m^{\tau*}L).$$
We let
$$d_m:=\deg (\mc{X}_m, \iota^*_m\mc{O}(1)_E)=\deg (\mc{X}_m,f_m^{*}L)=\deg (\mc{X}_m^{\tau},f_m^{\tau*}L)=\deg (\mc{X}_m^{\tau}, \iota_m^{\tau*}\mc{O}(1)_{E}).$$
Let $D(m,\tau)$ be the smallest integer greater than 
$$\max\left\{2(d_m^{-1/n}(n+3)/2 +d_m-2), 2(M(m,\tau)-n)(d_m-1)+n+2\right\}.$$
Let $v$ be the least common multiple of $1,2,\dots, [E:\QQ]$ and let
$$B_m:=\exp\left(\exp(33.3) \cdot 3d_m^{1/n}\frac{nv}{2(n+1)} \right).$$
By the projection formula, we have 
$$d_m=\deg (\mc{X}_m,f_m^*L)=|\Gamma_{p^m}|\deg L.$$
By \eqref{bounding monodromy group, large local system}, we then have
\begin{align*}
B_m
&<\exp\left(\exp(33.3) \cdot 3\log B\cdot p^{k^2/n}(\deg L)^{1/n}\frac{nv}{2(n+1)} \right)
\\&=B^{C_4},
\end{align*}
where $C_4$ is some constant depending only on $\mc{X}, \mf{T}$ and $T$.
By Theorem \ref{Arakelov Bombieri-Pila} with $\varepsilon=1$, for any 
$B'\geq \max\{B^{C_4},e\},$
the set of $E$-points $Q$ of $\mc{X}_m^{\tau}$ with exponential Arakelov height 
$H_{\iota_m^{\tau*}\ol{\mc{O}(1)}}(Q)\leq B'$
can be covered by 
$$d_mD(m,\tau)\cdot \poly_{n,E}(M(m,\tau),d_m)\cdot (M(m,\tau)+1)^{2[E:\QQ](n+1)d_m^{-1/n}}B'^{2(n+1)d_m^{-1/n}}$$
irreducible subvarieties over $E$ of dimensions $< \dim \mathcal{X}$ and degree at most $d_mD(m,\tau)$.

Suppose $P$ is an $\mathcal{O}_{E,T}$-point of $\mathscr{X}$ with height $H_{L}(P) \leq B$. 
The fiber of $\ms{X}_m\to \ms{X}$ above $P$ corresponds to an element in $H^1(\Gal(E_T/E), G_m)$.
Then by \cite[p. 22]{Sko01}, $f_m^{\tau}(Q)=P$ for some $\tau=[\rho]\in H^1(\Gal(E_T/E), G_m)$ and some $Q\in \mc{X}_m^{\tau}(E)$.
By \eqref{bounding monodromy group, large local system} and \eqref{conjectural height bound}, for some constants $C_5\geq 1$ and $C_6\geq C_4$ depending only on $\mc{X}, \mf{T}$ and $T$, we have
\begin{align*}
H_{\iota_m^{\tau*}\ol{\mc{O}(1)}}(Q)
&=\exp([E:\QQ] h_{\iota^{\tau*}_m\ol{\mc{O}(1)}}(Q))
\\&< \exp([E:\QQ](C_1+C_2\log B+ C_3|\Gamma_{p^m}|^{1/n}))
 \\&<C_5B^{C_6[E:\QQ]}.
\end{align*}
It follows that the set of $\mathcal{O}_{E,T}$-points of $\mathcal{X}$ with exponential height at most $B$, that are images of $E$-points in $\mc{X}_m^{\tau}$ under $f_{m}^{\tau}$ can be covered by 
\begin{align*}
\kappa(B):=& d_mD(m,\tau)\cdot \poly_{n,E}(M(m,\tau),d_m) \cdot (M(m,\tau)+1)^{2[E:\QQ](n+1)d_m^{-1/n}}
\\&\quad\quad\quad\quad\quad 
\cdot C_5^{2(n+1)d_m^{-1/n}} B^{C_6[E:\QQ]\cdot 2(n+1)d_m^{-1/n}}
\end{align*}
irreducible subvarieties over $E$ of dimensions $< \dim \mathcal{X}$ and degree at most $d_m^2D(m,\tau)$ with respect to $L$ (this assertion about the degree uses the isomorphism $f_m^{\tau*}L\simeq \iota^{\tau*}_m\mc{O}(1)_E$ and the projection formula).
We have 
$$M(m,\tau)\leq \poly_{\mc{X},\mf{T},T}(|\Gamma_{p^m}|)\leq \poly_{\mc{X},\mf{T},T}(\log B).$$
Similarly,
$$d_m= |\Gamma_{p^m}|\deg L\leq \poly_{\mc{X},\mf{T},T}(\log B).$$
Since 
$$d_m=|\Gamma_{p^m}|\deg L \geq (\log B)^{n}\deg L,$$ 
we have 
$$B^{d_m^{-1/n}}\leq B^{(\log B)^{-1}(\deg L)^{-1/n}}=O_{\mc{X},\mf{T},T}(1).$$
Similarly, 
$$C_5^{2(n+1)d_m^{-1/n}} = O_{\mc{X},\mf{T},T}(1)$$
and 
$$(M(m,\tau)+1)^{2[E:\QQ](n+1)d_m^{-1/n}}=O_{\mc{X},\mf{T},T}(1).$$
Therefore, $\kappa(B)$ is at most $\poly_{\mc{X},\mf{T},T}(\log B)$ when $B>e$.

By properness of $\mc{X}$, the set of $E$-points in $\mc{X}$ is in bijection with the set of $\mc{O}_{E,T}$-points on $\ms{X}$. 
The number of twists is at most the cardinality of $H^1(\Gal(E_T/E),G_m)$, which by assumption is at most $\poly_{\mc{X},\mf{T},T}(|\Gamma_{p^m}|)$, and thus at most $\poly_{\mc{X},\mf{T},T}(\log B)$.
Therefore, when $B>e$, the set of $E$-points in $\mc{X}$ with exponential height at most $B$ can be covered by $\poly_{\mc{X},\mf{T},T}(\log B)$ irreducible subvarieties over $E$ of dimensions $< \dim \mathcal{X}$ and degree at most $d_m^2D(m,\tau)\leq \poly_{\mc{X},\mf{T},T}(\log B)$ with respect to $L$.
By defining a universally fixed mapping $\mc{X}\mapsto (\mc{X},\mf{T},T)$, where $\mc{X}$ satisfies the conditions in our theorem, we have   $\poly_{\mc{X},\mf{T},T}(\cdot)=\poly_{\mc{X}}(\cdot)$ by convention.
\end{proof}

\section{Proof of Theorem \ref{abelian variety, height comparison}}\label{Proof of Neron's theorem}
Let $\mc{A}$ be an abelian variety of dimension $g$ over a number field $E$.
We will prove that $\mc{A}$ has a polylogarithmic-size covering of rational points by using Theorem \ref{main theorem}.
In Section \ref{Twists of finite etale covers}, we prove the polynomial bound for the size of the first Galois cohomology.
In Section \ref{Basis of global sections of twisted line bundles} and Section \ref{Projective embeddings and a height identity}, we construct integral models for the twists of the finite \'{e}tale covers and their projective embeddings, using Koizumi's theorem \cite{Koi76}, twisted trace and geometry of numbers.
We use a result by de Jong and Shokrieh \cite{dJS22} to express the difference between the N\'{e}ron-Tate height and the Arakelov height in terms of norms of sections with respect to canonical metrics introduced by Zhang \cite{Zha95}.
The result in geometry of numbers that we apply is an upper bound for the largest successive minima proved by de Boer \emph{et al.} \cite[Theorem A.3.1]{dBPW25}, which is an $L^{\infty}$ version of Bhargava \emph{et al.} \cite[Theorem 1.6]{BSTTTZ20}.
It allows us to obtain the desired bound for the norms of sections in Section \ref{Bounds for norms of sections} and Section \ref{Height comparison on twisted abelian varieties}.

\subsection{Twists of finite \'{e}tale covers}\label{Twists of finite etale covers}
Let $S$ be a set of finite primes of $E$ such that $\mc{A}$ spreads out to an abelian scheme $\ms{A}_{\mathcal{O}_{E,S}}$ over $\mathcal{O}_{E,S}$.
Let $p$ be a prime number which is not below any prime in $S$.
For each positive integer $m$, let $[p^m]:\ms{A}_{\mathcal{O}_{E,S}}\to \ms{A}_{\mathcal{O}_{E,S}}$ be the multiplication by $p^m$ map.
Let $f_m:\mc{A}\to \mc{A}$ be the restriction of $[p^m]$ to the generic fiber.
It is a finite \'{e}tale morphism of degree $p^{2mg}$.
Let $k:=2g+1$.
Fix an embedding $\bG_a^{2g}\hookrightarrow \bGL_{k}$ defined by 
$$(a_1,\dots, a_{2g})\mapsto I_{k}+\sum_{i=1}^{2g} a_i E_{i,2g+1},$$
where $I_{k}$ is the identity matrix and $E_{i,2g+1}$ is the matrix with 1 in the $(i,2g+1)$-entry and $0$ elsewhere.
Let $\Gamma$ be the image of $\mathbb{Z}^{2g}$ under this embedding.
The geometric deck group $G_m$ of $f_m$ is isomorphic to the image $\Gamma_{p^m}$ of $\Gamma$ under the mod $p^m$ map $\bGL_k(\ZZ)\rightarrow \bGL_k(\ZZ/p^m\ZZ)$.
The morphisms $f_m$ form a tower $\mf{T}$ with $\Gamma_{p^m}$-structure in the sense of Definition \ref{Gamma prime power structure}.
Let $T$ be the union of $S$ and the set of primes of $E$ above $p$.
The set $T$ is exceptional in the sense of Definition \ref{exceptional primes}.

Let $\theta_m:\Gal(\ol{E}/E)\to \Aut(G_m)$ be the action defined in Section \ref{Statements of results}.
There is a $\Gal(\ol{E}/E)$-equivariant group isomorphism $\mc{A}[p^m](\ol{E})\simeq G_m$ given by $x\mapsto t_x$, where $t_x$ is the translation by $x$.
Since this Galois action is unramified outside $T$, it induces a $\Gal(E_T/E)$-action on $\mc{A}[p^m](\ol{E})\simeq G_m$,
where $E_T$ is the maximal extension of $E$ unramified outside $T$.
Let $G_{E,T}:= \Gal(E_T/E)$.
Let $\theta'_m: G_{E,T}\to \Aut(G_m)$ be the action. 
Every continuous cocycle $\rho: G_{E,T}\to G_m$ induces a homomorphism 
$$\theta_{m,\rho}:G_{E,T}\to G_m\rtimes \Aut(G_m)$$
given by $\sigma\mapsto (\rho(\sigma), \theta'_m(\sigma))$.
By Hermite-Minkowski \cite[Theorem 2.13]{Neu99}, there are only finitely many such homomorphisms, so $H^1(G_{E,T},G_m)$ has finite cardinality.

\begin{lemma}\label{Galois cohomology bound}
The cardinality of $H^1(G_{E,T}, G_m)$ is at most $\poly_{\mc{A},\mf{T}, T}(|G_m|)$.
\end{lemma}

\begin{proof}
We have the following short exact sequence of $G_{E,T}$-modules:
$$0\to \mc{A}[p^{m-1}](\ol{E})\to \mc{A}[p^m](\ol{E})\to \mc{A}[p](\ol{E})\to 0,$$
which induces an exact sequence
$$H^1(G_{E,T},\mc{A}[p^{m-1}](\ol{E}))\to H^1(G_{E,T},\mc{A}[p^m](\ol{E}))\to H^1(G_{E,T},\mc{A}[p](\ol{E})).$$
The cardinality $c$ of $H^1(G_{E,T},\mc{A}[p](\ol{E}))$ depends only on $\mc{A}$, $T$ and $p$.
Write $c=p^r$, where $r\geq 0$. 
By induction, the cardinality of $H^1(G_{E,T},\mc{A}[p^m](\ol{E}))$ is at most $c^m=p^{rm}=(p^{2mg})^{r/2g}\leq \poly_{\mc{A},T,p}(p^{2mg})$.
The lemma then follows from the $G_{E,T}$-equivariant isomorphism $\mc{A}[p^m](\ol{E})\simeq G_m$ above.
\end{proof}

Let $E_{m,\rho}$ be the fixed field of the kernel of $\theta_{m,\rho}$.

\begin{lemma}\label{bound of degree of extension}
We have $[E_{m,\rho}:E]\leq |G_m|^{2g+1}$.
\end{lemma}

\begin{proof}
Since $G_m\simeq (\ZZ/p^m\ZZ)^{2g}$, we have $\Aut(G_m)\simeq \tb{GL}_{2g}(\ZZ/p^m\ZZ)$.
Hence, $[E_{m,\rho}:E]\leq |G_m||\Aut(G_m)|\leq p^{2mg}\cdot p^{4g^2m}=p^{2mg(1+2g)}$.
\end{proof}

Let $\tau=[\rho]$ be a class in $H^1(G_{E,T}, G_m)$,
let $\mc{A}^{\tau}$ be the corresponding \emph{twist} given by \cite[Theorem 4.5.2]{Poonen17}.
By \cite[Remark 4.5.4]{Poonen17}, there exists an isomorphism $\lambda: \mc{A}_{E_T}\to \mc{A}^{\tau}_{E_T}$ such that 
\begin{equation}\label{twist and conjugation}
\lambda^{\sigma}:=\sigma_{\mc{A}^{\tau}}\circ\lambda\circ\sigma_{\mc{A}}^{-1}=\lambda\circ \rho(\sigma)
\end{equation} 
for any $\sigma\in \Gal(E_T/E)$, where $\sigma_{\mathcal{A}}: \mc{A}_{E_T}\to \mc{A}_{E_T}$ is the morphism induced by $\sigma^{-1}$.
For any $\sigma\in \Gal(E_T/E_{m,\rho})$, we have $\lambda^{\sigma}=\lambda\circ \rho(\sigma)=\lambda$,
so $\lambda$ is defined over $E_{m,\rho}$ by Galois descent.
Let $f_m^{\tau}:\mc{A}^{\tau}\ra \mc{A}$ be the descended finite morphism characterized by $(f_m^{\tau})_{E_{m,\rho}}\circ \lambda = (f_m)_{E_{m,\rho}}$.
Let $\ol{\rho}:\Gal(E_{m,\rho}/E)\to G_m$ be the continuous cocycle induced by $\rho$.
It satisfies the cocycle condition saying that for any $\sigma,\eta\in \Gal(E_{m,\rho}/E)$,
\begin{equation}\label{cocycle condition}
\ol{\rho}(\sigma\eta)=\ol{\rho}(\sigma)\cdot \theta_m(\sigma)(\ol{\rho}(\eta)).
\end{equation}

\subsection{Basis of global sections of twisted line bundles}\label{Basis of global sections of twisted line bundles}
Let $\mc{K}$ be a symmetric ample line bundle on $\mc{A}$, i.e. $(-1)^*\mc{K}=\mc{K}$.
Let $\mc{M}:= (f_m^*\mc{K})^{\otimes 3}$, which is very ample \cite[\S 17]{Mum70}.
Let $V:= H^0(\mc{A}_{E_{m,\rho}}, \mc{M}_{E_{m,\rho}})$.
For any $\sigma \in G:=\Gal(E_{m,\rho}/E)$, let $\sigma_{\mc{A}}: \mc{A}_{E_{m,\rho}}\to \mc{A}_{E_{m,\rho}}$ be the morphism induced by $\sigma^{-1}$.
Let $p: \mc{A}_{E_{m,\rho}}\to \mc{A}$ be the natural projection.
Since $\mc{M}$ is defined over $E$, we have 
$$\sigma^*_{\mc{A}}\mc{M}_{E_{m,\rho}}=\sigma^*_{\mc{A}}p^*\mc{M}=(p\circ\sigma_{\mc{A}})^*\mc{M}=p^*\mc{M}=\mc{M}_{E_{m,\rho}}.$$
We can then define a Galois action $G\times V\to V$ by $\sigma\cdot s= (\sigma^{-1}_{\mc{A}})^*s$ for any $\sigma\in G$ and $s\in V$.
This Galois action defines for any $\sigma$ a $\sigma$-semilinear automorphism on $V$.
Let $\mc{M}^{\tau}:=(f_m^{\tau*}\mc{K})^{\otimes 3}$. We have an isomorphism 
$$\lambda^*: H^0(\mc{A}^{\tau}_{E_{m,\rho}}, \mc{M}^{\tau}_{E_{m,\rho}})\simeq V.$$
This identifies $W^{\tau}:= H^0(\mc{A}^{\tau}, \mc{M}^{\tau})$ with a subspace $V^{G,tw}$ of $V$ fixed by a twisted Galois action defined by 
$$\delta_{\sigma}(s):=\sigma\cdot_{tw} s:= (\lambda^{-1}\circ \sigma^{-1}_{\mc{A}^{\tau}}\circ \lambda)^*(s)$$
for any $\sigma\in G$ and $s\in V$.

\begin{lemma}\label{twisted trace composition}
We have $\delta_{\sigma}=(\sigma^{-1}_{\mc{A}})^*(\ol{\rho}(\sigma^{-1}))^*$ and we have $\delta_{\eta}\circ\delta_{\sigma}=\delta_{\eta\sigma}$ for all $\eta,\sigma\in G$.
\end{lemma}

\begin{proof}
An equation similar to \eqref{twist and conjugation} holds for any $\sigma\in G$, so for any $s\in V$, 
$$\delta_{\sigma}(s)=(\lambda^{-1}\circ \sigma^{-1}_{\mc{A}^{\tau}}\circ \lambda)^*(s)=(\ol{\rho}(\sigma^{-1})\circ \sigma^{-1}_{\mc{A}})^*(s)=(\sigma^{-1}_{\mc{A}})^*(\ol{\rho}(\sigma^{-1}))^*(s).$$
By \eqref{cocycle condition}, we have
$$\ol{\rho}((\eta\sigma)^{-1})=\ol{\rho}(\sigma^{-1}\eta^{-1})=\ol{\rho}(\sigma^{-1})\cdot\theta_m(\sigma^{-1})(\ol{\rho}(\eta^{-1})).$$
By the definition of $\theta_m$, we have
$$\theta_m(\sigma^{-1})(\ol{\rho}(\eta^{-1}))=\sigma^{-1}_{\mc{A}}\circ \ol{\rho}(\eta^{-1})\circ \sigma_{\mc{A}}.$$
Therefore, for any $s\in V$,
\begin{align*}
\delta_{\eta}(\delta_{\sigma}(s))
&=(\eta^{-1}_{\mc{A}})^*(\ol{\rho}(\eta^{-1}))^*(\sigma^{-1}_{\mc{A}})^*(\ol{\rho}(\sigma^{-1}))^*(s)
\\&=(\ol{\rho}(\sigma^{-1})\sigma^{-1}_{\mc{A}}\ol{\rho}(\eta^{-1})\eta^{-1}_{\mc{A}})^*(s)
\\&=(\ol{\rho}(\sigma^{-1})\cdot \theta_m(\sigma^{-1})(\ol{\rho}(\eta^{-1}))\cdot\sigma_{\mc{A}}^{-1}\eta^{-1}_{\mc{A}})^*(s)
\\&=(\ol{\rho}((\eta\sigma)^{-1})\cdot(\eta\sigma)_{\mc{A}}^{-1})^*(s)
\\&=\delta_{\eta\sigma}(s).
\end{align*}
The lemma follows.
\end{proof} 

Define the twisted trace $\text{Tr}_\tau: V\to V$ by 
$$\text{Tr}_\tau(s)=\sum_{\sigma\in G}\delta_{\sigma}(s).$$

\begin{lemma}\label{twisted trace surjection}
The twisted trace $\text{Tr}_\tau$ maps $V$ onto $V^{G,tw}$.
\end{lemma}

\begin{proof}
For any $\eta\in G$, Lemma \ref{twisted trace composition} gives
$$\delta_{\eta}(\text{Tr}_\tau(s))=\sum_{\sigma\in G}\delta_{\eta}\delta_{\sigma}(s)=\sum_{\sigma\in G}\delta_{\eta\sigma}(s)=\sum_{\kappa\in G}\delta_{\kappa}(s)=\text{Tr}_\tau(s),$$
so $\text{Tr}_\tau(s)\in V^{G,tw}$.
If $u\in V^{G,tw}$, then 
$$\text{Tr}_\tau\left(\frac{u}{|G|}\right)=\sum_{\sigma\in G}\delta_{\sigma}\left(\frac{u}{|G|}\right)=\sum_{\sigma\in G}\frac{u}{|G|}=u.$$
Therefore, $\text{Tr}_\tau$ is surjective.
\end{proof}

We replace $|D_K|^{1/2}$ in a standard result in the geometry of numbers by $|D_K|^{1/n}$:

\begin{lemma}\label{Minkowski bound}
Let $K$ be a number field of degree $n=[K:\QQ]$ with discriminant $D_K$.
There exists a $\QQ$-basis $\{\gamma_1,\dots, \gamma_n\}$ for $K$ consisting of algebraic integers $\gamma_i\in \mc{O}_K$ such that $|\gamma_i|_v\leq |D_K|^{1/n}$ for all Archimedean places $v$ of $K$.
\end{lemma}

\begin{proof}
The field $K$ embeds into 
$$K_{\RR}:=\left\{(x_v)\in \prod_{v:K\hookrightarrow\CC} \CC: \ol{x_v}=x_{\ol{v}}\right\}$$
via its Archimedean embeddings.
The volume of the fundamental domain of $\mc{O}_K$ as a lattice is $\sqrt{|D_K|}$, see \cite[Prop. 5.2]{Neu99}.
Let $\lambda_i$ be the $i$-th successive minimum of the centrally symmetric convex set
$$X:=\{(x_v)\in K_{\RR}: |x_v|\leq 1 \text{ for all } v\}$$
with respect to $\mc{O}_K$,
i.e. $\lambda_i$ is the minimum of the real numbers $\lambda>0$ such that $\lambda X$ contains at least $i$ linearly independent vectors of $\mc{O}_K$.
By \cite[Theorem A.3.1]{dBPW25} ($L^{\infty}$ version of \cite[Theorem 1.6]{BSTTTZ20}), we have $\lambda_n\leq |D_K|^{1/n}$.
The lemma follows.
\end{proof}

Fix a basis $\mf{B}:=\{b_0,\dots, b_N\}$ of global sections of $\mc{K}^{\otimes 3}$.
By Koizumi's theorem \cite{Koi76}, the multiplication map
\begin{equation}\label{Koizumi}
H^0(\mc{A}, \mc{K}^{\otimes 3})^{\otimes p^{2m}}\ra H^0(\mc{A}, \mc{K}^{\otimes 3p^{2m}})
\end{equation}
is surjective, so we can construct a basis $\mf{A}:=\{a_0,\dots, a_{M(m)}\}$ of $H^0(\mc{A}, \mc{K}^{\otimes 3p^{2m}})$ such that each $a_i$ is some degree $p^{2m}$ monomial in $b_0,\dots, b_N$. 
The isomorphism $f_m^*\mc{K}^{\otimes 3}\simeq \mc{K}^{\otimes 3p^{2m}}$ induces an isomorphism
$$\Psi: H^0(\mc{A}, \mc{K}^{\otimes 3p^{2m}})\simeq H^0(\mc{A}, f_m^*\mc{K}^{\otimes 3}).$$
We have a basis $\mf{U}:=\{u_0,\dots, u_{M(m)}\}$ for $W:=H^0(\mc{A}, f_m^*\mc{K}^{\otimes 3})$, 
where $u_i=\Psi(a_i)$.

By Lemma \ref{Minkowski bound}, there exists a $\QQ$-basis $\{\alpha_1,\dots, \alpha_d\}$ for $E_{m,\rho}$ consisting of algebraic integers $\alpha_i\in \mc{O}_{E_{m,\rho}}$ such that $|\alpha_i|_v\leq |D_{E_{m,\rho}}|^{1/d}$ for all Archimedean places $v$ of $E_{m,\rho}$.
The set
$$\{\alpha_au_i: 1\leq a\leq d, ~0\leq i\leq M(m)\}$$
spans $V=W_{E_{m,\rho}}$ over $E$.
By Lemma \ref{twisted trace surjection}, the image of this set under the twisted trace $\text{Tr}_{\tau}$ spans $V^{G,tw}$ over $E$. 
Since $\lambda^*(W^{\tau})=V^{G,tw}$, we have 
$$\dim V^{G,tw}=\dim W^{\tau}=\dim W^{\tau}_{E_{m,\rho}}=\dim V=\dim W =M(m)+1.$$
There exists $v_j:=\text{Tr}_{\tau}(\alpha_{a_j}u_{i_j})$, where $j=0,\cdots, M(m)$, such that $\{v_0,\dots v_{M(m)}\}$ forms a basis for $V^{G,tw}$ over $E$.
Let $t_i:= (\lambda^*)^{-1}(v_i)$.
Then $\{t_0,\dots, t_{M(m)}\}$ forms a basis for $W^{\tau}$ over $E$.

\subsection{Projective embeddings and a height identity}\label{Projective embeddings and a height identity}
Let $\iota^{\tau}_E: \mc{A}^{\tau}\hookrightarrow \mb{P}^{M(m)}_{E}$ be the embedding defined by the basis $\{t_0,\dots, t_{M(m)}\}$ in $W^{\tau}:= H^0(\mc{A}^{\tau},\mc{M}^{\tau})$.
Let $\ms{A}^{\tau}_{\mc{O}_E}$ be the Zariski closure of the image of $\mc{A}^{\tau}$ in $\mb{P}^{M(m)}_{\mc{O}_E}$.
Let $\iota^{\tau}: \ms{A}^{\tau}_{\mc{O}_E}\hookrightarrow \mb{P}^{M(m)}_{\mc{O}_E} $ be the embedding.
Let $\ol{\ms{M}^{\tau}}$ be the Hermitian line bundle on $\ms{A}^{\tau}_{\mc{O}_E}$ defined as $\iota^{\tau*}\ol{\mc{O}(1)}$ equipped with the Fubini-Study metrics.
For any line bundle $\mc{N}$ on $\mc{A}$, denote by $\widehat{h}_{\mc{N}}$ the N\'{e}ron-Tate height on $\mc{A}$ with respect to $\mc{N}$.
Let $h_{\ol{\ms{M}^{\tau}}}$ be the logarithmic Arakelov height on $\mc{A}^{\tau}$ with respect to $\ol{\ms{M}^{\tau}}$.

Let $e: \Spec E\to \mc{A}$ be the zero section. 
A rigidification of a line bundle $\mc{N}$ on $\mc{A}$ is an $E$-linear isomorphism $\kappa: E\simeq e^*\mc{N}$.
We rigidify $\mc{K}$ by choosing a non-zero element $z$ in the fiber of $\mc{K}$ above the origin and define $\kappa: E\simeq e^* \mc{K}$ by sending $1$ to $z$.
For every embedding $w: E\hookrightarrow \CC$, let $\mc{A}_w$ and $\mc{K}_w$ be the base changes of $\mc{A}$ and $\mc{K}$ to $\CC$ via $w$ respectively.
For every non-Archimedean place $w$ of $E$, we fix a completion $\widehat{\ol{E}_w}$ of $\ol{E}_w$, and set $\mc{A}_w=\mc{A}_{\widehat{\ol{E}_w}}$ and $\mc{K}_w=\mc{K}_{\widehat{\ol{E}_w}}$.
For each place $w$ of $E$, the analytification $\mc{K}_w^{an}$ is equipped with the canonical metric $\|\cdot\|_{\mc{K},w}$ in the sense of Zhang \cite{Zha95}, and de Jong and Shokrieh \cite[Section 5]{dJS22}.
It satisfies the property that $[n]^*\mc{K}\simeq \mc{K}^{\otimes n^2}$ is an isometry with respect to the canonically induced metrics on $[n]^*\mc{K}$ and $\mc{K}^{\otimes n^2}$ for all positive integers $n$, see \emph{loc. cit}.
Let $s$ be a non-zero rational section of $\mc{K}$ such that $P\in \mc{A}(E)$ does not lie on the support of $\text{div}_{\mc{K}}(s)$.
Write $n_w:=[E_w: \QQ_{w_0}]$, where $w_0$ is the place below a place $w$.
By \cite[(14.4)]{dJS22}, we have 
\begin{equation}\label{NT height and canonical metrics}
\widehat{h}_{\mc{K}}(P)=-\frac{1}{[E:\QQ]}\sum_{w\in M_E}n_w\log \|s(P)\|_{\mc{K},w}.
\end{equation}
We equip $\mc{M}^{\tau}:=(f_m^{\tau*}\mc{K})^{\otimes 3}$ with the metrics induced from that on $\mc{K}$.
These are indeed the canonical metrics on $\mc{M}^{\tau}$ since $[n]^*\mc{K}\simeq \mc{K}^{\otimes n^2}$ is an isometry and since we have the isomorphism $\lambda: \mc{A}_{E_{m,\rho}}\simeq \mc{A}^{\tau}_{E_{m,\rho}}$ with $(f_m^{\tau})_{E_{m,\rho}}\circ\lambda=(f_m)_{E_{m,\rho}}$.

\begin{lemma}\label{Arakelov and NT height identity}
Let $P\in \mc{A}^{\tau}(E_{m,\rho})$. Then
\begin{align*}
h_{\ol{\ms{M}^{\tau}}}(P)
= \widehat{h}_{\mc{M}^{\tau}_{E_{m,\rho}}}(P)
&+\frac{1}{[E_{m,\rho}:\QQ]}\sum_{v\nmid\infty}n_v\log \max_{0\leq i\leq M(m)}\|t_i(P)\|_{\mc{M}^{\tau}_{E_{m,\rho}},v}
\\&\quad+\frac{1}{[E_{m,\rho}:\QQ]}\sum_{v\mid\infty} n_v\log\left(\sum_{i=0}^{M(m)}\|t_i(P)\|^2_{\mc{M}^{\tau}_{E_{m,\rho}},v}\right)^{1/2}.
\end{align*}
\end{lemma}

\begin{proof}
Choose a non-zero rational section $s$ of $\mc{M}^{\tau}_{E_{m,\rho}}$ such that $s(P)\neq 0$.
For every $0\leq i\leq M(m)$, there is a unique $x_i\in E_{m,\rho}$ such that $t_i(P)=x_is(P)$.
Since $t_0,\dots, t_{M(m)}$ have no common zeros, $x_i$ is non-zero for some $i$, and $\iota^{\tau}_E(P)=[x_0,\dots, x_{M(m)}]$.
By the definition of logarithmic Arakelov height, 
\begin{align*}
h_{\ol{\ms{M}^{\tau}}}(P)
&=h_{\ol{\mc{O}(1)}}([x_0,\dots, x_{M(m)}])
\\&=\frac{1}{[E_{m,\rho}:\QQ]}\left(\sum_{v\nmid \infty } n_v\log \max_{0\leq i\leq {M(m)}}|x_i|_v+\sum_{v\mid\infty} n_v\log \left(\sum_{i=0}^{M(m)}|x_i|_v^2\right)^{1/2}\right)
\\&= -\frac{1}{[E_{m,\rho}:\QQ]}\sum_{v\in M_{E_{m,\rho}}}n_v \log \|s(P)\|_{\mc{M}^{\tau}_{E_{m,\rho}},v}
\\&\quad\quad+\frac{1}{[E_{m,\rho}:\QQ]}\sum_{v\nmid \infty } n_v\log \max_{0\leq i\leq {M(m)}}\|t_i(P)\|_{\mc{M}^{\tau}_{E_{m,\rho}},v}
\\&\quad\quad+\frac{1}{[E_{m,\rho}:\QQ]}\sum_{v\mid\infty} n_v\log \left(\sum_{i=0}^{M(m)}\|t_i(P)\|_{\mc{M}^{\tau}_{E_{m,\rho}},v}^2\right)^{1/2}.
\end{align*}
By applying \eqref{NT height and canonical metrics} to $\mc{M}^{\tau}_{E_{m,\rho}}$, the lemma follows.
\end{proof}

The classical comparison of N\'{e}ron-Tate height and Weil height on embedded abelian varieties goes back to Zarhin-Manin \cite{ZM72}, where the embedding in their work is defined using Mumford coordinates. 
Lemma \ref{Arakelov and NT height identity} above considers the basis $\{t_0,\dots, t_{M(m)}\}$ of global sections constructed in Section \ref{Basis of global sections of twisted line bundles} instead.

\subsection{Bounds for norms of sections}\label{Bounds for norms of sections}
For any place $w$ of $E$, rigidified line bundle $\mc{L}$, and section $s\in H^0(\mc{A}, \mc{L})$, we define 
$$\|s\|_{\sup,w}:= \sup_{x\in \mc{A}^{an}_w}\|s(x)\|_{\mc{L}, w},$$
where $\|\cdot\|_{\mc{L},w}$ is the canonical metric defined as in Section \ref{Projective embeddings and a height identity}.
Let $\mf{B}:=\{b_0,\dots, b_N\}$ be the basis for $ H^0(\mc{A}, \mc{K}^{\otimes 3})$ fixed in Section \ref{Basis of global sections of twisted line bundles}. 
Let 
$$C_w:=\max\{1, \max_{0\leq k\leq N}\|b_k\|_{\sup,w}\}.$$
Let $\iota_{\mc{A}}:\mc{A}\hookrightarrow \mb{P}^N_E$ be the embedding defined by the basis $\mf{B}$.
Let $\ms{A}$ be the Zariski closure of $\ms{A}_{\mathcal{O}_{E,S}}$ in $\mb{P}^N_{\mc{O}_E}$.
Let $\iota_{\ms{A}}:\ms{A}\hookrightarrow \mb{P}^N_{\mc{O}_E}$ be the embedding.
Let $\ol{\ms{N}}$ be the Hermitian line bundle on $\ms{A}$ defined as $\iota^*_{\ms{A}}\ol{\mc{O}(1)}$ equipped with the Fubini-Study metrics.

\begin{lemma}\label{finite log sup norm sum}
The quantity $\displaystyle c_{\mf{B}}:= \frac{1}{[E:\QQ]}\sum_{w\in M_E}n_w \log C_w$ is finite. 
\end{lemma}

\begin{proof}
The abelian variety $\mc{A}$ has good reduction outside $S$. 
If $w$ is a non-Archimedean place of $E$ outside $S$, 
then as mentioned in \cite[Section 5.1]{dJS22}, the canonical metric on $\mc{K}^{\otimes 3}_w$ is equal to the model metric $\|\cdot \|_{\ms{N}_{\mc{O}_{E_w}}}$ induced by $(\ms{A}_{\mc{O}_{E_w}},\ms{N}_{\mc{O}_{E_w}})$, which is characterized in \cite[Section 4.1]{dJS22} by the following property. 
Let $s$ be a non-zero rational section of $\mc{K}_w^{\otimes 3}$ viewed as a rational section of $\ms{N}_{\mc{O}_{E_w}}$.
Let $x\in \mc{A}_w^{an}$.
Let $\xi$ be the reduction of $x$ on the special fiber of the model $\ms{A}_{\mc{O}_{E_w}}$.
Let $U$ be an open neighborhood of $\xi$ in $\ms{A}_{\mc{O}_{E_w}}$ such that $\ms{N}_{\mc{O}_{E_w}}$ is trivialized on $U$. 
Let $t$ be a trivializing element of $\ms{N}_{\mc{O}_{E_w}}(U)$.
Let $f$ be the unique rational function on $U$ satisfying $s=f\cdot t$ on $U$.
Then we set $\|s(x)\|_{\ms{N}_{\mc{O}_{E_w}}}=|f(x)|$, which is at most $1$ since $f(x)\in \mc{O}_{E_w}$.
Hence, $C_w=1$ for all $w\not\in S$.
At each of the remaining places, $C_w$ is finite because $\mc{A}^{an}_w$ is compact.
The lemma follows.
\end{proof}

\begin{lemma}\label{norms of twisted sections}
The basis $\mf{A}:=\{a_0,\dots, a_{M(m)}\}$ for $H^0(\mc{A}, \mc{K}^{\otimes 3p^{2m}})$ and the basis $\mf{U}:=\{u_0,\dots, u_{M(m)}\}$ for $W:=H^0(\mc{A}, f_m^*\mc{K}^{\otimes 3})$ defined in Section \ref{Basis of global sections of twisted line bundles} satisfy $\|a_i\|_{\sup,w}\leq C_w^{p^{2m}}$ and $\|u_i\|_{\sup,w}\leq C_w^{p^{2m}}$ for any place $w$ of $E$ and any $i=0,\dots, M(m)$.
For any $\delta_{\sigma}$ as in Section \ref{Basis of global sections of twisted line bundles}, where $\sigma\in \Gal(E_{m,\rho}/E)$, any $\alpha\in E_{m,\rho}$, any $i=0,\dots, M(m)$, and any place $v$ of $E_{m,\rho}$,
$$\|\delta_{\sigma}(\alpha u_i)\|_{\sup, v}\leq |\sigma(\alpha)|_v C^{p^{2m}}_{w(v)},$$
where $w(v)$ is the place of $E$ below $v$.
\end{lemma}

\begin{proof}
Since $a_i$ is a product of $p^{2m}$ sections among $b_k$, we have $\|a_i\|_{\sup,w}\leq C_w^{p^{2m}}$.
By the property of canonical metric, $\Psi$ defined in Section \ref{Basis of global sections of twisted line bundles} is an isometry.
Since $u_i=\Psi(a_i)$, we have $\|u_i\|_{\sup,w}=\|a_i\|_{\sup,w}\leq C_w^{p^{2m}}$.

By Lemma \ref{twisted trace composition}, we have 
$\delta_{\sigma}=(\sigma^{-1}_{\mc{A}})^*(\ol{\rho}(\sigma^{-1}))^*$.
Let $\ol{\mc{M}}_{\text{can}}$ and $\ol{\mc{K}}^{\otimes 3}_{\text{can}}
$ be the line bundles $\mc{M}$ and $\mc{K}^{\otimes 3}$ equipped with the canonical metrics respectively.
Since $[p^m]\circ \ol{\rho}(\sigma^{-1})=[p^m]$, we have 
$$\ol{\rho}(\sigma^{-1})^*\ol{\mc{M}}_{\text{can}}=\ol{\rho}(\sigma^{-1})^*[p^m]^*\ol{\mc{K}}^{\otimes 3}_{\text{can}}=[p^m]^*\ol{\mc{K}}^{\otimes 3}_{\text{can}}=\ol{\mc{M}}_{\text{can}}.$$
Hence $\ol{\rho}(\sigma^{-1})^*$ acts isometrically.
Pullback under $\sigma^{-1}_{\mc{A}}$ of the canonical metric on $\mc{M}_{(E_{m,\rho})_{\sigma^{-1}v}}$ is the canonical metric on $\mc{M}_{(E_{m,\rho})_{v}}$.
Therefore, $\|\delta_{\sigma}(u_i)\|_{\sup, v}=\|u_i\|_{\sup,\sigma (v)}=\|u_i\|_{\sup, w(v)}\leq C_{w(v)}^{p^{2m}}$.
By semilinearity, we have $\|\delta_{\sigma}(\alpha u_i)\|_{\sup, v}=|\sigma(\alpha)|_v\|\delta_{\sigma}(u_i)\|_{\sup,v}\leq |\sigma(\alpha)|_vC_{w(v)}^{p^{2m}}$.
\end{proof}

We emphasize that we are able to take the $1/[E_{m,\rho}:\QQ]$-th root of $|D_{E_{m,\rho}}|$ in the following lemma because of the result by de Boer \emph{et al.} \cite[Theorem A.3.1]{dBPW25} (an $L^{\infty}$ version of Bhargava \emph{et al.} \cite[Theorem 1.6]{BSTTTZ20}).

\begin{lemma}\label{norm bound for twisted sections}
Let $\{v_0,\dots, v_{M(m)}\}$ and $\{t_0,\dots, t_{M(m)}\}$ be as in Section \ref{Basis of global sections of twisted line bundles}.
Let $v$ be a place of $E_{m,\rho}$ and let $w(v)$ be the place of $E$ below $v$.
Let $D_{E_{m,\rho}}$ be the discriminant of $E_{m,\rho}$.
For any $Q\in \mc{A}(E_{m,\rho})$, and any $j=0,\dots, M(m)$,
\begin{equation}
\|t_j(\lambda(Q))\|_{\mc{M}^{\tau}_{E_{m,\rho}},v}\leq \|v_j\|_{\sup,v}\leq 
\begin{cases}
C^{p^{2m}}_{w(v)},
&  v\nmid \infty;
\\ [E_{m,\rho}:E] |D_{E_{m,\rho}}|^{1/[E_{m,\rho}:\QQ]}C^{p^{2m}}_{w(v)},
&v\mid \infty.
\end{cases}
\end{equation}
\end{lemma}

\begin{proof}
By definition, 
$$\displaystyle v_j:=\text{Tr}_{\tau}(\alpha_{a_j}u_{i_j})=\sum_{\sigma\in \Gal(E_{m,\rho}/E)}\delta_{\sigma}(\alpha_{a_j}u_{i_j}).$$ 
Since $\alpha_{a_j}$ is an algebraic integer, we have $|\sigma(\alpha_{a_j})|_v\leq 1$ for any $v\nmid\infty$.
For any $v\mid\infty$, by Lemma \ref{Minkowski bound}, we have 
$$|\sigma(\alpha_{a_j})|_v=|\alpha_{a_j}|_{\sigma^{-1}v}\leq |D_{E_{m,\rho}}|^{1/[E_{m,\rho}:\QQ]}.$$
We then obtain the desired upper bound for $\|v_j\|_{\sup,v}$ using Lemma \ref{norms of twisted sections} and the (resp. ultrametric) triangle inequality in the Archimedean (resp. non-Archimedean) case.
The first inequality follows from $t_j:= (\lambda^*)^{-1}(v_j)$.
\end{proof}

\subsection{Height comparison on twisted abelian varieties}\label{Height comparison on twisted abelian varieties} 

We first estimate several terms which will appear in the proof of the height comparison:

\begin{lemma}\label{discriminant bound}
We have $\log [E_{m,\rho}:E]\ll_g\log p^m$, $\log (M(m)+1)\ll_{\mc{A},\mc{K}}\log p^m$,
and $$\frac{\log |D_{E_{m,\rho}}|}{[E_{m,\rho}:\QQ]} \ll_{\mc{A},T}\log p^m.$$
\end{lemma}

\begin{proof}
By Lemma \ref{bound of degree of extension}, we have $\log [E_{m,\rho}:E]\leq \log p^{2mg(2g+1)}\ll_g\log p^m$.
The surjection in \eqref{Koizumi} factors through the surjection
$$Sym^{p^{2m}}H^0(\mc{A},\mc{K}^{\otimes 3})\to H^0(\mc{A}, \mc{K}^{\otimes 3p^{2m}}).$$
Let $r$ be the dimension of $H^0(\mc{A},\mc{K}^{\otimes 3})$. Then
\begin{equation}\label{Koizumi and M(m) bound}
M(m)+1=\dim H^0(\mc{A},f_m^*\mc{K}^{\otimes 3})\leq \binom{p^{2m}+r-1}{r-1}\leq (p^{2m}+1)^{r-1},
\end{equation}
so $\log (M(m)+1)\ll_{\mc{A},\mc{K}}\log p^m$.
By  \cite[III. Theorem 2.6]{Neu99}, if $\ell$ is a prime number dividing $D_E$ or lying below a prime in $T$, 
then the $\ell$-adic valuation $v_{\ell}(|D_{E_{m,\rho}}|)$ is at most 
$$[E_{m,\rho}:\QQ]-1+[E_{m,\rho}:\QQ]\log ([E_{m,\rho}:\QQ])/\log \ell\ll_{T}[E_{m,\rho}:\QQ]\log [E_{m,\rho}:\QQ].$$
Let $\ell_0$ be the largest prime among the primes dividing $D_E$ or lying below a prime in $T$.
Since $E_{m,\rho}$ is unramified outside $T$, we have 
$$\frac{\log |D_{E_{m,\rho}}|}{[E_{m,\rho}:\QQ]}\ll_{T}(\log [E_{m,\rho}:\QQ])\cdot \log \ell_0\ll_{\mc{A},T}1+\log [E_{m,\rho}:E]\ll_{\mc{A},T}\log p^m.$$
\end{proof}

By Lemma \ref{Galois cohomology bound}, to complete the proof of Theorem \ref{abelian variety, height comparison}, it suffices to prove the following: 

\begin{thm}\label{height comparison, abelian}
For any $Q\in \mc{A}(E_{m,\rho})$,  we have
$$h_{\ol{\ms{M}^{\tau}}}(\lambda(Q))\leq \widehat{h}_{\mc{M}^{\tau}_{E_{m,\rho}}}(\lambda(Q))+O_{\mc{A},T, \mc{K}}((\deg f_m)^{1/g}).$$
The tower $\mf{T}$ formed by the morphisms $f_m$ satisfies height comparison on twists.
\end{thm}

\begin{proof}
For any place $w$ of $E$, we have $\displaystyle \sum_{v\mid w}n_v =n_w[E_{m,\rho}:E]$, so
$$\frac{1}{[E_{m,\rho}: \QQ]}\sum_{v\in M_{E_{m,\rho}}}n_v \log C_{w(v)}=\frac{1}{[E:\QQ]}\sum_{w\in M_E}n_w \log C_w=c_{\mf{B}}.$$
Recall that $\deg f_m=p^{2mg}$. By Lemma \ref{Arakelov and NT height identity}, Lemma \ref{norm bound for twisted sections} and Lemma \ref{discriminant bound},
\begin{align*}
h_{\ol{\ms{M}^{\tau}}}(\lambda(Q))
&= \widehat{h}_{\mc{M}^{\tau}_{E_{m,\rho}}}(\lambda(Q))
+\frac{1}{[E_{m,\rho}:\QQ]}\sum_{v\nmid\infty}n_v\log \max_{0\leq i\leq M(m)}\|t_i(\lambda(Q))\|_{\mc{M}^{\tau}_{E_{m,\rho}},v}
\\&\quad\quad\quad\quad\quad\quad+\frac{1}{[E_{m,\rho}:\QQ]}\sum_{v\mid\infty} n_v\log\left(\sum_{i=0}^{M(m)}\|t_i(\lambda(Q))\|^2_{\mc{M}^{\tau}_{E_{m,\rho}},v}\right)^{1/2}
\\&\leq \widehat{h}_{\mc{M}^{\tau}_{E_{m,\rho}}}(\lambda(Q))+p^{2m}c_{\mf{B}}+\log [E_{m,\rho}:E]+\frac{\log |D_{E_{m,\rho}}|}{[E_{m,\rho}:\QQ]}+\frac{1}{2}\log (M(m)+1)
\\&\leq \widehat{h}_{\mc{M}^{\tau}_{E_{m,\rho}}}(\lambda(Q))+O_{\mc{A},T, \mc{K}}((\deg f_m)^{1/g}).
\end{align*}

Take $L$ to be $\mc{K}^{\otimes 3}$, which depends only on $\mc{A}$.
By \eqref{Koizumi and M(m) bound}, we have $M(m)\leq \poly_{\mc{A}}((\deg f_m)^{1/g})$.
We have $f_m^{\tau*}L=\mc{M}^{\tau}\simeq \iota^{\tau*}_E\mc{O}(1)$.
For any $B>e$ and $P\in \mc{A}^{\tau}(E)$ with exponential Weil height $H_{L}(f^{\tau}_m(P))\leq B$, we have  
\begin{align*}
h_{\ol{\ms{M}^{\tau}}}(P)
&\leq \widehat{h}_{\mc{M}^{\tau}}(P)+ O_{\mc{A},T}((\deg f_m)^{1/g})
\\&=\widehat{h}_{L}(f_m^{\tau}(P))+ O_{\mc{A},T}((\deg f_m)^{1/g})
\\&<C_1+\log B+ C_2(\deg f_m)^{1/g},
\end{align*}
where $C_1$ and $C_2$ are some constants depending only on $\mc{A}$ and $T$.
\end{proof}

\end{document}